\documentclass[11pt]{amsproc}

\usepackage{amsmath,amsthm,amsfonts,graphicx,amssymb,amscd}
\usepackage{graphicx}
 \usepackage{epstopdf}

\newtheorem{thm}{Theorem}[section]

\newtheorem{lem}[thm]{Lemma}

\theoremstyle{definition}
\newtheorem{defn}[thm]{Definition}

\numberwithin{equation}{section}

\begin{document}

\title[Asymptotic behavior of complex earthquakes and Teich. disks]{On the asymptotic behavior of complex earthquakes and Teichm\"{u}ller disks}


\author{Subhojoy Gupta}
\address{Center for Quantum Geometry of Moduli Spaces, Ny Munkegade 118, DK 8000 Aarhus C, Denmark.}
\curraddr{California Institute of Technology, Mathematics 253-37, Pasadena CA 91125 USA}
\email{subhojoy@caltech.edu}
\thanks{This work was (partly) supported by the Danish Research Foundation DNRF95 (Center for Quantum Geometry of Moduli Spaces - QGM)}


\date{\today}

\begin{abstract}
Given a hyperbolic surface and a simple closed geodesic on it, complex-twists along the curve produce a holomorphic family of deformations in Teichm\"{u}ller space,  degenerating to the Riemann surface where it is pinched. We show there is a corresponding Teichm\"{u}ller disk such that the two are strongly asymptotic, in the Teichm\"{u}ller metric, around the noded Riemann surface. We establish a similar comparison with plumbing deformations that open the node.

\end{abstract}

\maketitle

\section{Introduction}
Let $S$ be a closed oriented surface of genus $g\geq 2$. The Teichm\"{u}ller space $\mathcal{T}_g$ is the space of marked Riemann surfaces of genus $g$, or equivalently, marked hyperbolic structures on $S$.  By a result of Bers this space embeds as a bounded domain in a finite dimensional complex space, and acquires a complex structure and a non-degenerate Kobayashi metric. This \textit{Teichm\"{u}ller metric} $d_\mathcal{T}$ is complete and any complex geodesic admits a description in terms of linear maps in the singular-flat structure induced by a holomorphic quadratic differential (see \S 2). In this article we consider certain  one-complex-parameter families of deformations in $\mathcal{T}_g$ that arise from its hyperbolic and conformal descriptions, and examine their asymptotic behavior.\\

Given a hyperbolic surface $X$ and a geodesic loop $\gamma$, a \textit{complex earthquake} or  \textit{complex-twist } deforms by a ``twist"  and ``graft"  along $\gamma$, determined, respectively,  by the real and imaginary parts  of a complex parameter $z\in {\mathbb{H}_{\geq 0}}=\mathbb{H} \cup \mathbb{R}$. This produces a holomorphic immersion $\mathcal{E}_\gamma: {\mathbb{H}_{\geq 0}} \to \mathcal{T}_g$ such that $\mathcal{E}_\gamma(0) = X$. On the other hand, given a Riemann surface $Y$ and a \textit{Jenkins-Strebel differential} associated with $\gamma$ one has a totally-geodesic \textit{Teichm\"{u}ller disk} $\mathcal{D}_\gamma: \mathbb{H}  \to \mathcal{T}_g$ that maps $i\mapsto Y$ and deforms by a ``shear"  and ``stretch"  of the corresponding metric cylinder. We provide descriptions of both in \S2.\\

In this article we prove the following asymptoticity:

\begin{thm}\label{thm:thm1}
Given a hyperbolic surface $X$ and simple closed geodesic $\gamma$ there is a Riemann surface $Y$ such that the corresponding complex-earthquake deformations $\mathcal{E}_\gamma$ and Teichm\"{u}ller disk $\mathcal{D}_\gamma$ as above are strongly asymptotic in a horodisk around $\infty \in \bar{\mathbb{H}}$. That is, for any $\epsilon>0$ there is an $H>0$ such that
\begin{equation*}
d_\mathcal{T}(\mathcal{E}_\gamma(z), \mathcal{D}_\gamma(z)) <\epsilon
\end{equation*}
for all $z\in \mathbb{H}$ satisfying $Im(z) > H$. Similarly, given a Teichm\"{u}ller disk  $\mathcal{D}_\gamma$ determined by a Jenkins-Strebel differential associated to $\gamma$, there is a complex earthquake disk $\mathcal{E}_\gamma$ that is asymptotic to it as above.
\end{thm}

\begin{figure}
  \centering
  \includegraphics{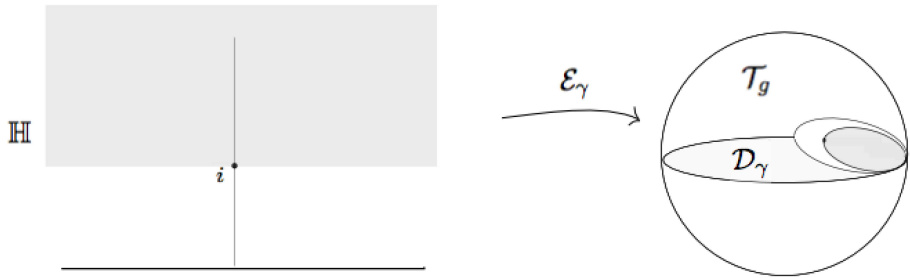}\\
  \caption{A cartoon illustrating the asymptoticity in Theorem \ref{thm:thm1}.}
\end{figure}

As $Im(z) \to \infty$ the conformal structures limit, in either family of deformations, to a noded Riemann surface $\Sigma_\infty$ in the \textit{augmented Teichm\"{u}ller space} $\widehat{\mathcal{T}_g}$. One can regenerate Riemann surfaces in $\mathcal{T}_g$ by the procedure of \textit{plumbing} that glues two punctured disks in a neighborhood $U$  of the node. Apart from the choice of $U$, this also involves one complex ``gluing" parameter and produces a map $\mathcal{P}: \mathbb{H} \cup\{ \infty\} \to \widehat{\mathcal{T}_g}$ such that $\mathcal{P}(\infty) = \Sigma_\infty$ (see \S 5.1). We shall also prove:

\begin{thm}\label{thm:thm2}
For any plumbing disk  $\mathcal{P}$ as above  there exists a Teichm\"{u}ller disk $\mathcal{D}$ such that the two are strongly asymptotic in a horodisk around $\infty \in \bar{\mathbb{H}}$. Moreover, there is a choice of a neighborhood $U_0$ of the node such that the corresponding plumbing disk  $\mathcal{P}_0$ coincides with $\mathcal{D}$. 
\end{thm}

These results permit a comparison of complex-twists, Teichm\"{u}ller disks, and plumbing near the noded surface. The complex earthquake deformations considered here in fact extend to a larger domain that defines a \textit{complex earthquake disk} that is a   \textit{proper} embedding of the extended domain in $\mathcal{T}_g$ (see \cite{McM}, \cite{EMM}). A fuller investigation of the asymptotic behaviour  at the other limit points of the image, together with the more general case of complex earthquakes along a geodesic \textit{lamination},  are left for subsequent study.\\

Theorem \ref{thm:thm1} is an extension of the multicurve-case of the asymptoticity result in \cite{Gup1}. There, we consider \textit{grafting rays}, which are complex earthquakes in the ``purely imaginary" direction.  Our exposition here hopes to clarify the holomorphic context in which that result really lies. The proof considered the \textit{Thurston metric} associated with a grafted surface and involved constructing quasiconformal maps to appropriate singular-flat surfaces along the Teichm\"{u}ller geodesic ray. The strategy in this article is the same, with the additional consideration of ``shears" or ``twists". The techniques extend easily to provide a proof of Theorem \ref{thm:thm2}.

\bigskip

\textbf{Acknowledgements.} I wish to thank Chris Judge for an inspiring conversation, Yair Minsky for his continued encouragement, and the organizers of the conference at Almora for their invitation.

\section{Background}

For basic references for the expository account in this section see, for example, \cite{Hubb} and \cite{ImTan}.\\

\subsubsection*{Notation}

$\mathbb{H} = \{z \in \mathbb{C} \vert Im(z) >0\}$, $\mathbb{H}_{>T} = \{z \in \mathbb{C} \vert Im(z) >T\}$, $\bar{\mathbb{H}} = \mathbb{H}_{\geq 0} \cup \{\infty\}$. The latter is equipped with the \textit{horocycle topology}, namely, neighborhoods of $\infty$ are the horodisks $\mathbb{H}_{>T}$.

\subsection{Augmented Teichm\"{u}ller space}

The \textit{Teichm\"{u}ller space} for the surface $S$ is the collection of pairs:

\begin{center}
$\mathcal{T}_g = \{(f, \Sigma) \vert$ $ f:S_{g}\to \Sigma$ is a homeomorphism,\\ $\Sigma$ is a Riemann surface $\}$/$\sim$
\end{center}
where the equivalence relation is:
\begin{center}
$(f, \Sigma) \sim  (g, \Sigma^\prime)$
\end{center}
if there is a conformal homeomorphism $h:\Sigma \to \Sigma^\prime$ such that $h\circ f$ is isotopic to $g$.\\

\textbf{Pinching.}  A noded surface $(Z,P)$ is a surface with a set $P$ of distiguished points such that $Z\setminus P$ is a Riemann surface of finite type, and each point in $P$ has an open neighborhood $U_P$ biholomorphic to 
\begin{center}
$\{(z,w)\in \mathbb{C}^2 \vert$ $zw=0$, $\lvert z\rvert$, $\lvert w\rvert <1\}$.
\end{center}
where $P$ maps to $(0,0)\in \mathbb{C}^2$. In this article we shall assume that $P$ consists of a single point. Note that $U_P\setminus P$ is biholomorphic to a pair of punctured disks.\\

\begin{figure}[h]
  \centering
  \includegraphics{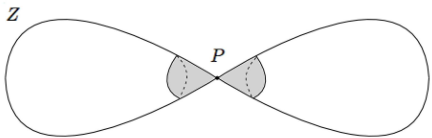}\\
  \caption{A noded surface.}
\end{figure}

Let $\gamma$ be a \textit{loop}, that is, a simple closed curve on $S$ non-trivial in homotopy. A point $(f,\Sigma)$ in Teichm\"{u}ller space admits a \textit{pinching} map $P_\gamma:(\Sigma, \gamma) \to (Z,P)$ to a noded surface that collapses $\gamma$ to $P$. In the Fenchel-Nielsen coordinates, this amounts to setting the length $l(\gamma) =0$.\\

\textbf{The augmented space.} 
For each surface $X\in \mathcal{T}_g$, consider the countably infinite collection $\mathcal{C}(X)$ of noded surfaces obtained by pinching simple closed multicurves, and form the \textit{augmented Teichm\"{u}ller space}:
\begin{equation*}
\widehat{\mathcal{T}_g} = \mathcal{T}_g \cup \{  \mathcal{C}(X) \vert X\in \mathcal{T}_g \} 
\end{equation*}
where note that a noded surface in any $\mathcal{C}(X)$ comes with a marking that is the post-composition of the marking on $X$ with the corresponding pinching map.\\
The topology on $\mathcal{T}_g$ can be extended to the augmented space as follows: for every choice of neighborhood $U_P$ of the node on a surface $Y\in \mathcal{C}(X)$, define the open set:
\begin{center}
$N(Y;U_P, \epsilon)$ = $\{(f, R) \in \mathcal{T}_g\vert$ $\exists$ an annular neighborhood $A$ of the pinching (multi)curve on $R$ and a $(1+\epsilon)$-quasiconformal map $h:R\setminus A \to Y\setminus U_P$ preserving the marking $\}$.
\end{center}
 \medskip
These form a neighborhood basis around $P$. See \S1.3 of \cite{Abi}, for example.

\subsection{Complex earthquakes}

Our exposition shall follow that \cite{McM}, we refer to those papers for details. Since we shall restrict to the case of simple closed curves, these can also be called ``complex twists". (For the general case of a measured geodesic lamination, see for example \cite{EMM}.)\\




\subsubsection*{The real case} Consider a hyperbolic surface $X$ and a simple closed geodesic $\gamma$. A  (real) \text{twist} deformation of $X$ is the one-real-parameter family of hyperbolic surfaces $\mathcal{E}_\gamma(t,X)$ ($t\geq 0$)  obtained by cutting along $\gamma$, and gluing back by a rotation by distance $t\in \mathbb{R}$.  As usual, the sign gives the ``direction" of the twist as determined by the orientation of the surface. In particular, twisting by $t=l_X(\gamma)$ is equivalent to a Dehn-twist: any transverse arc now loops around $\gamma$ an additional time in the positive direction.
 \\

In the universal cover $\widetilde{X} = \mathbb{H}$, identify the vertical axis with a lift  of $\gamma$. The twist at this lift  corresponds to the transformation:

\begin{equation}\label{eq:exp}
  z\mapsto \begin{cases}
    e^{t}z, & \text{if $z\in \mathbb{H}^-$}.\\
    z, & \text{if $z\in \mathbb{H}^+$}.
  \end{cases}
\end{equation}
where $\mathbb{H}^\pm = \{z\in \mathbb{H} \vert \pm Re(z) >0\}$. This local model of the map extends equivariantly to all other lifts. \\

\begin{figure}
  \centering
  \includegraphics{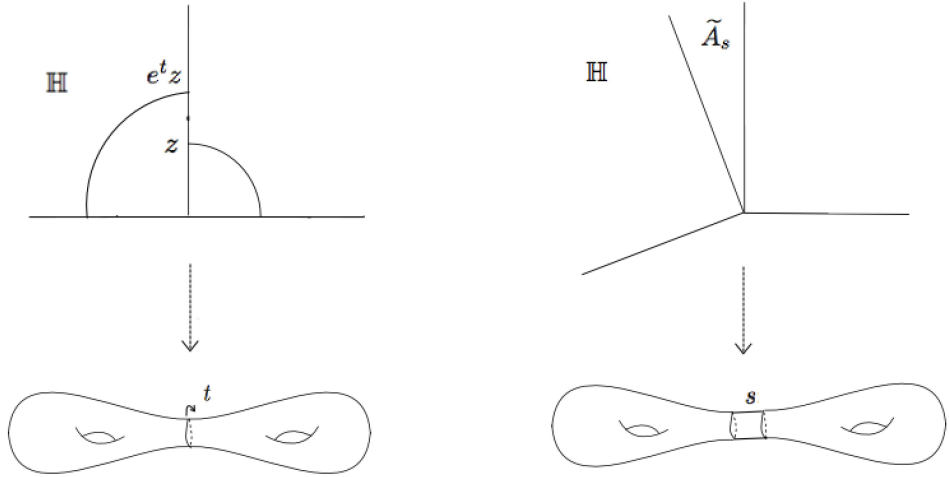}\\
  \caption{Complex twists: a real $t$-twist (left), and $is$-graft (right).}
\end{figure}

\subsubsection*{Grafting} This operation of ``purely imaginary twisting" cuts along $\gamma$ and inserts an annulus. The local model for the map in the universal cover is:

\begin{equation}
  z\mapsto \begin{cases}
    e^{is}z, & \text{if $z\in \mathbb{H}^-$}.\\
    z, & \text{if $z\in \mathbb{H}^+$}.
  \end{cases}
\end{equation}

and the sector of angle $s$:\\
\begin{center}
$\widetilde{A}_s = \{ z\in \mathbb{C}^\ast \vert arg(z) \in [\frac{\pi}{2}, \frac{\pi}{2} + s]\}$
\end{center}
\medskip
is inserted in the space between (see Figure 3). This sector is invariant under the infinite-cyclic subgroup $\langle \gamma \rangle \in PSL_2(\mathbb{R})$ corresponding to $[\gamma] \in \pi_1(X)$. On the surface, this amounts to grafting in the quotient annulus, and the resulting surface is $\mathcal{E}_\gamma(is,X)$.\\

Finally, for a parameter $z = t + is$, a \textit{complex twist}  amounts to composing a twist of amount $t$ by a ``graft" of amount $s$:
\begin{equation}
\mathcal{E}_{\gamma}(t + is, X) = \mathcal{E}_\gamma(is, \mathcal{E}_\gamma(t,X))
\end{equation}
\medskip

\textit{Remark.} In what follows,  $\mathcal{E}_{\gamma}(t + is)$ shall denote $\mathcal{E}_{\gamma}(t + is, X)$ whenever the choice of hyperbolic surface $X$ is understood.\\

\textbf{Thurston metric.} By the previous description, a complex-twist deforms the complex-projective structure on $X$.  On the universal cover $\widetilde{X_t}$ of the resulting surface, we can define a \textit{projective metric} as follows. For $x\in \widetilde{X_t}$ and $v\in T_x\widetilde{X_t}$, the length of the vector is defined to be:\\
\begin{equation*}
t(v) = \inf\limits_{i:\mathbb{D}\to \widetilde{X_t}}  \rho(i^\ast v)
\end{equation*}
where $i$ is a projective immersion such that $x\in i(\mathbb{D})$, and $\rho$ is the Poincar\'{e} metric on $\mathbb{D}$. (See \S2.1 of \cite{Tan}, and \cite{KulPink} for a more general context.)\\
On the grafted regions it is in fact euclidean: in the local model the metric on the sector $\widetilde{A}$ is given by the expression $\frac{\lvert dz\rvert}{\lvert z\rvert}$ that is easily seen to be flat (see \cite{Tan}). In its complement, the metric remains hyperbolic.\\
By the M\"{o}bius invariance of the metric, this descends to a hybrid euclidean-and-hyperbolic metric on the surface called the \textit{Thurston metric}. In particular, the grafted annulus $\widetilde{A}/ \langle \gamma \rangle$ in the quotient is a euclidean cylinder of length $t$. The complementary region is isometric to $X\setminus \gamma$. 

\subsection{Quadratic differentials, flat surfaces and Teichm\"{u}ller disks}

A holomorphic (\textit{resp.} meromorphic) \textit{quadratic differential} on a Riemann surface $\Sigma$ is a $(2,0)$-tensor  locally expressed as $q(z) dz^2$, where $q(z)$ is a holomorphic (\textit{resp.} meromorphic) function. \\
There is a change of coordinates $z\mapsto \xi $ in which the local expression for such a differential is $d\xi^2$. Locally, this is a branched covering of $\mathbb{C}$, and the pullback of the euclidean metric defines a global singular flat metric on $\Sigma$. The foliations by horizontal and vertical lines pull back to the \textit{horizontal} and \textit{vertical foliation}, respectively, on $\Sigma$. The metric has a local ``pronged" structure at the zeroes. A pole of order $2$ with residue $\alpha \in \mathbb{R}$ has a neighborhood isometric to a half-infinite euclidean cylinder of circumference $2\pi \alpha$ (see \cite{Streb}).\\

Conversely, a \textit{flat surface} obtained by identifying parallel sides of an embedded planar polygon by translations or semi-translations  (\textit{i.e}, a translation following by a reflection $z\mapsto -z$) acquires a holomorphic structure  and a quadratic differential induced from the restriction of $dz^2$ to the planar region.\\
Such a planar polygon determines  a cyclically-ordered collection of vectors $\vec{v}_1, \vec{v}_2, \ldots \vec{v}_{2k}$ in  $\mathbb{R}^2$ corresponding to the $2k$ sides, by taking the differences of its endpoints. The above pairing of parallel sides implies that this collection determines at most $k$ directions (\textit{i.e.} vectors upto sign). Conversely, such a collection, together with the pairing data for the sides, determines a flat surface upto a global translation. \\





\textbf{The $SL_2(\mathbb{R})$ action.} The \textit{Teichm\"{u}ller distance} between two points $X$ and $Y$ in $\mathcal{T}_{g}$ is:
\begin{equation}\label{eq:dt}
d_{\mathcal{T}}(X,Y) = \frac{1}{2}\inf\limits_{f} \ln K_f
\end{equation}
where $f:X\to Y$ is a quasiconformal homeomorphism preserving the marking and $K_f$ is its quasiconformal dilatation. See \cite{Ahl} for definitions. The fact that this coincides with the Kobayashi metric on $\mathcal{T}_g$ was shown in (\cite{Royd2}).\\

 Given a flat surface $F$ and  $A\in SL_2(\mathbb{R})$ we can define a new flat surface $A\cdot F$ by the linear action: that is, given by the new collection of vectors $\{A\cdot \vec{v}_1, A\cdot \vec{v}_2, \ldots A\cdot \vec{v}_{2k}\}$ where $F$ is described by the $2k$-tuple of vectors, as before. The image of this orbit is totally geodesic in  $\mathcal{T}_g$ with respect to the Teichm\"{u}ller metric, and since an elliptic rotation does not change the conformal structure, this defines an embedding of  $SL_2(\mathbb{R})/SO(2,\mathbb{R})= \mathbb{H}$ in $\mathcal{T}_g$. \\

In particular, the real-parameter family 
\begin{equation*}
s\mapsto \bigl(\begin{smallmatrix} 1/\sqrt{s}  &0\\ 0& \sqrt{s} \end{smallmatrix} \bigr)  \cdot X
\end{equation*}
is a geodesic ray. (This is the image of the imaginary axis of $\mathbb{H}$ in the identification above.) These transformations have the additional property that they preserve the euclidean \textit{area}. We shall, however, prefer to preserve horizontal lengths by a conformal rescaling of a factor $\sqrt{s}$:

\begin{defn}[Stretches and shears]\label{defn:stsh}  A \textit{stretch} by amount $s$ is the rescaled map:
\begin{equation*}
X\mapsto   \bigl(\begin{smallmatrix} 1&0\\ 0&s \end{smallmatrix} \bigr) \cdot X
\end{equation*}
and a \textit{shear} by an amount $t$ transforms  the flat surface:
\begin{equation*}
X\mapsto   \bigl(\begin{smallmatrix} 1&t\\ 0&1 \end{smallmatrix} \bigr) \cdot X, 
\end{equation*}
where in both cases, the pairing data for the sides remains the same.
\end{defn}

\begin{defn}[Teichm\"{u}ller disk] Given $t+is \in \mathbb{H}$ and a flat surface $X$, a $t$-twist followed by an $s$-stretch determines a surface:
\begin{equation}
\mathcal{D}_\gamma(t + is, X) = \bigl(\begin{smallmatrix} 1 &0\\ 0&s \end{smallmatrix} \bigr) \cdot  \bigl(\begin{smallmatrix} 1&t\\ 0&1 \end{smallmatrix} \bigr) \cdot X
\end{equation} 
where, as usual, the pattern of identifications of the sides remains the same. This defines the \textit{Teichm\"{u}ller disk} $\mathcal{D}_\gamma: \mathbb{H} \to \mathcal{T}_g$ with \textit{basepoint} $X$. As usual we shall drop the second argument when the choice of $X$ is understood.
\end{defn}

\textbf{An example.} Let $S$ be a torus. The Teichm\"{u}ller space $\mathcal{T}_1 = \mathbb{H}$:  any $\tau \in \mathbb{H}$ defines a torus with a complex structure $X_\tau := \mathbb{C}/(\mathbb{Z} \oplus \tau \mathbb{Z})$ together with the marking $\phi:S\to X_\tau$ induced by the linear action $\tilde{\phi}:\mathbb{C}\to \mathbb{C}$ such that $\tilde{\phi}(1)=1$ and $\tilde{\phi}(i)=\tau$. The set of simple closed curves are parametrized by $p/q \in \mathbb{Q}\cup \{\infty\}$, and the augmented space $\widehat{\mathcal{T}_1}$ is obtained by adding the rational points in $\partial \mathbb{H} = \mathbb{R}\cup \{\infty\}$.\\
Let $X:= X_i$ be the square torus, and $\mathcal{D}_\gamma$ the Teichm\"{u}ller disk with basepoint $X$. The image by a ``stretch by amount $s$"  $Y := \mathcal{D}_\gamma(is, X)$ is then a rectangle with sides  $s$ and $1$. The Gr\"{o}tzsch argument then yields $d_\mathcal{T} (X, Y) = \ln s = d_\mathbb{H}(i, is)$. The image by a  ``shear by amount $t$" $\mathcal{D}_\gamma(t, X)$ is a parallelogram with two horizontal sides and two sides of slope $t$, with parallel sides identified. For $t\in \mathbb{Z}$ this corresponds to $t$ Dehn-twists around the  $1/0$-curve (see Lemma \ref{lem:ttwist}). The Teichm\"{u}ller disk in this case coincides with the entire Teichm\"{u}ller space $\mathbb{H}$. 

\section{Preliminaries}

\subsection{Jenkins-Strebel surface, and twists.} 

\begin{defn}\label{defn:jss} A \textit{Jenkins-Strebel surface} is a flat surface $F$  given by the induced metric of a holomorphic quadratic differential whose horizontal foliation has all closed leaves. Such an $F$ can be obtained by taking a euclidean cylinder and gluing the two boundary components by some interval-exchange map. As a flat surface, $F$ then has exactly two sides that are parallel, in what we shall call the ``vertical" direction, and all other sides are parallel or semi-parallel in an orthogonal ``horizontal" direction. 
\end{defn}

\textit{Remark.} It is a theorem of Jenkins and Strebel (\cite{Streb}) that given a Riemann surface and any non-trivial simple closed curve, such a holomorphic quadratic differential exists with the closed horizontal leaves in the homotopy class of that curve.

\begin{defn} The \textit{height} of a Jenkins-Strebel surface is the length of its vertical sides, that is, the  euclidean distance between the horizontal sides of the euclidean cylinder.
\end{defn}

\begin{figure}
  \centering
  \includegraphics{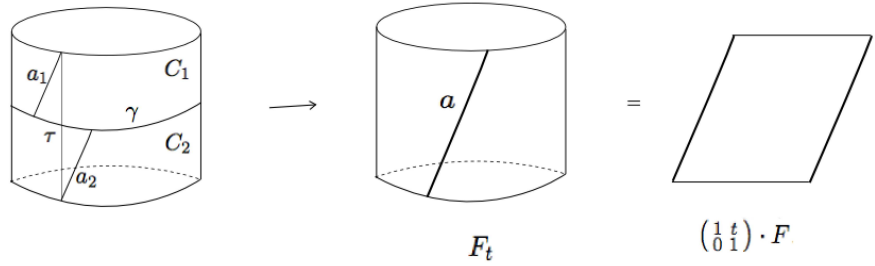}\\
  \caption{A $t$-twist amounts to shearing the Jenkins-Strebel surface (Lemma \ref{lem:ttwist}).}
\end{figure}

We shall use the following terminology:

\begin{defn}\label{defn:ttwist} A \textit{t-twist} is the  operation of cutting along the core-curve of a metric cylinder (embedded in some surface) and gluing back after a positive twist by a distance $t$.
\end{defn}

Note that a flat torus is a Jenkins-Strebel surface. The observation in the previous torus-example can be generalized as follows:

\begin{lem}\label{lem:ttwist} Let $F$ be a Jenkins-Strebel surface of height $1$. Then a $t$-twist is equivalent to a shear by amount $t$. Namely, if $F_t$ is the surface obtained by a $t$-twist, then it is identical to the surface $F^\prime =  \bigl(\begin{smallmatrix} 1&t\\ 0&1 \end{smallmatrix} \bigr) \cdot F$.
\end{lem}
\begin{proof}
Consider the euclidean cylinder $C$ obtained by cutting $F$ along the horizontal sides. Choose a vertical arc $\tau$ across $C$, such that $C\setminus \tau$ is a rectangle $R$.  Let $\gamma$ be the the horizontal circle on $C$  along which we cut and twist. Let $C_1$ and $C_2$ be  the components of $C\setminus \gamma$. $\tau$ restricts to the vertical arcs $\tau_1$ and $\tau_2$ on these subcylinders.\\
On $F_t$ the endpoints of $\tau_1$ and $\tau_2$ are displaced by a distance $t$. Now consider the arcs $a_1$ and $a_2$ of slope $1/t$ on $C_1$ and $C_2$ respectively, that each have one endpoint common with $\tau_1$ and $\tau_2$ (see Figure 4). On $F_t$ the arcs $a_1$ and $a_2$ line up to form an arc $a$ of slope $t$. Cutting $F_t$ along $a$ produces a parallelogram that is the image of $R$ when sheared by an amount $t$, and the lemma follows.
\end{proof}

\subsection{Bounded twisting} 

Let $C$ be a half-infinite euclidean cylinder. This admits a ``horizontal" foliation by  circles, and a transverse ``vertical" foliation by perpendicular straight rays.\\

In what follows, we shall often identify such a $C$ with the punctured unit disk $\mathbb{D}^\ast $ via a conformal identification that maps $\infty$ to $0$.  (Note that any two such identifications differ by a rotation, a freedom we shall employ later.) The horizontal leaves are mapped to circles centered at the origin, and the vertical leaves are mapped to radial segments. The euclidean metric is given by the expression $\frac{\lvert dz\rvert}{\lvert z\rvert}$ on $\mathbb{D}^\ast$. Throughout, $C_L$ shall denote the euclidean subcylinder of length $L$ that is adjacent to $\partial C$.

\begin{defn} Let $C$ be a half-infinite cylinder and $\tau$ be any vertical leaf. A continuous map $f:C\to C$ that is a homeomorphism to its image is said to be \textit{eventually twist-free} if  there exists an $L_0>0$ such that for any $L>L_0$, the segment $\tau \cap (C_L\setminus C_{L_0})$ can be homotoped to be disjoint from its image under $f$,  by a homotopy that fixes its endpoints. 
\end{defn}

\textit{Remark.} It can be checked that this definition is independent of the choice of vertical leaf.

\begin{lem}\label{lem:bt} Let $C$ be a half-infinite cylinder, and $g:C\to C$ be a conformal map that is a homeomorphism to its image. Then $g$ is eventually twist-free.
\end{lem}
\begin{proof}
By the above conformal identification $g$ can can be thought of as a conformal map from $\mathbb{D}^\ast$ into itself, and can be extended to the missing point to give a conformal embedding $g:\mathbb{D}\to \mathbb{D}$.\\
Milder assumptions involving just the existence and continuity of the derivative at the origin now suffice, and to emphasize that  we shall work in polar coordinates $(r,\theta)$ on $\mathbb{D}$, with respect to which
\begin{equation*}
g(r,\theta) =  (\alpha(r,\theta), \beta(r,\theta))
\end{equation*}
and has a continuous partial derivatives in $\mathbb{D}$.\\
 \begin{figure}
  \centering
  \includegraphics{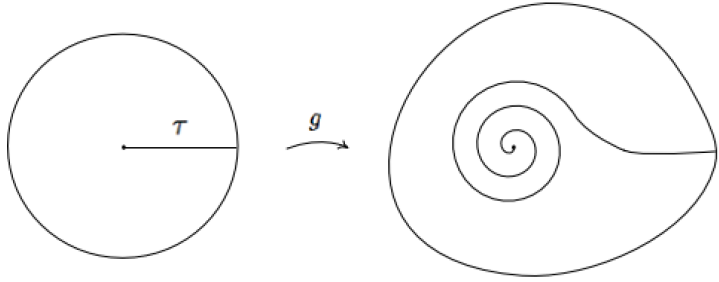}\\
  \caption{ A smooth univalent map $g$ twists finitely many times around the origin (Lemma \ref{lem:bt}).}
\end{figure}

Consider the image of any radial segment $\tau$.  For each ``twist" this arc would have to make an additional essential intersection with $\tau$, going around the origin once.\\
By continuity $ \frac{\partial \beta}{\partial r}(r,0)  \to  \frac{\partial \beta}{\partial r}(0,0) = \alpha  $ as $r\to 0$. For sufficiently small $r$ we have $\left\vert  \frac{\partial \beta(r,0)}{\partial r} \right\vert < 2\alpha$ and consequently the total change of the angle coordinate:
\begin{equation*}
\left\vert \int\limits_0^r \frac{\partial \beta(r,0)}{\partial r} \right\vert < 2\alpha r.
\end{equation*}

 In particular, there exists some $r_0$ such that for $r<r_0$ the above integral has value less than $2\pi$, that is, the image of the sub-segment $[0,r_0]\subset \tau$ does not make any further ``twists" about the origin. In the metric cylinder, the ``twist-free" subdisk $B_{r_0}$ corresponds to the complement of the subcylinder $C_{L_0}$, where  $L_ 0 = \ln \frac{1}{r_0}$.
\end{proof}

 \subsection{Quasiconformal lemmata}
\begin{lem}[Dehn-twist]\label{lem:untwist}

For any $\epsilon>0$ there is an $M_0>0$ such that any annulus $A$ of modulus greater than $M_0$ admits a $(1+\epsilon)$-quasiconformal map $f:A\to A$ fixing the boundary pointwise, that is homotopic to a single Dehn-twist.
\end{lem}
\begin{proof}
By uniformizing, we can assume $A$ is a euclidean cylinder of circumference $1$ and height $H>M_0$, which can be thought of as a euclidean rectangle $R$ with vertical sides identified. An argument identical to that of Lemma \ref{lem:ttwist} shows that a single Dehn-twist is realized by the linear map  $\bigl(\begin{smallmatrix} 1&t/H\\ 0&1 \end{smallmatrix} \bigr)$. For $M_0$ (and consequently $H$) sufficiently large, this map is  $(1+\epsilon)$-biLipschitz and hence $(1+\epsilon)$-quasiconformal.
\end{proof}


The following lemma is a modified version of Lemma 5.1 of \cite{Gup1}:

\begin{lem}[Interpolation]\label{lem:interp}
Let $g:\mathbb{D}\to\mathbb{C}$ be a univalent conformal map such that $g(0)=0$ and $g^\prime(0)=c$. Then for any $\epsilon>0$ there is a (sufficiently small)  $r>0$ and a $(1+\epsilon)$-quasiconformal map $f:\mathbb{D}\to g(\mathbb{D})$ that restricts to the dilatation $z\mapsto cz$ on $B_r$ and agrees with $g$ on $\mathbb{D}\setminus B_{2r}$. 
\end{lem}
\begin{proof}[Sketch of the proof]
Let $g(z) = cz + \psi(z)$, and consider a suitable bump function $\phi_r(z)$ that is identically $1$ on $\mathbb{D} \setminus B_{2r}$ and $0$ on $B_r$. It can be checked that for sufficiently small $r$, the map $f = cz + \phi_r(z)\psi(z)$ works: we refer to the proof of Lemma 5.1 in \cite{Gup1} for details. A key ingredient is the Koebe distortion theorem that controls the behaviour of $g$ close to $0$. (Note that the proof in \cite{Gup1} assumes $g^\prime(0)=1$, which can easily be removed by the rescaling $z\mapsto cz$.) 
\end{proof}

\subsection{Conformal limits}

\begin{defn}[Conformal limit]\label{defn:climit}
Let $\Sigma_n$ ($n\geq 0$) be a sequence of marked Riemann surfaces such that $l_{\sigma_n}(\gamma) \to 0$ where $\gamma$ is  the geodesic representative of a fixed homotopy class of a loop and  $\sigma_n$ is the corresponding sequence of uniformizing hyperbolic metrics. A noded  Riemann surface $Z$ with node $P$ is said to be a \textit{conformal limit} of the sequence if for any $\epsilon>0$ there is an $N>0$ and a sequence of $(1+\epsilon)$-quasiconformal maps
\begin{center}
 $f_{\epsilon,n}:\Sigma_n \setminus \gamma  \to Z\setminus U_n$
 \end{center}
for all $n\geq N$, such that these maps preserve marking, and the neighborhoods $U_n$ of the node shrink to $P$. \\
(Equivalently, $\Sigma_n \to (Z,P)$  in the topology on $\widehat{\mathcal{T}_g}$ mentioned in \S2.1.)
\end{defn}

\medskip
Recall that along a grafting ray $X_s$ ($s\geq 0$) for a loop $\gamma$ the grafted euclidean cylinder gets longer, and we get a limiting Riemann surface:

\begin{defn}[$\text{X}_\infty$] The surface $X_\infty$ is obtained by cutting along the geodesic representative of $\gamma$ on the hyperbolic surface $X=X_0$ and gluing half-infinite euclidean cylinders along the resulting two boundary components. The hyperbolic and euclidean metric is a $C^1$ metric and $X_\infty$ is a noded Riemann surface $(Z,P)$ (with the node $P$ at $\infty$ along the paired cylinders).
\end{defn}

The surface $X_\infty$ is the conformal limit of the grafting ray (see \cite{Gup3}). The following lemma shows that this limit is well defined if one considers other ``diverging"  directions as well:

\begin{lem}\label{lem:elim}
Let $z_n \in \mathbb{H}$ ($n\geq 0$) be any sequence such that $z_n\to \infty$  in the horocycle topology. Then $Z = X_\infty \cup \{P\}$ is the conformal limit of the sequence of Riemann surfaces $\mathcal{E}_\gamma(z_n)$.
\end{lem}
\begin{proof}
Recall that $X_n = \mathcal{E}_\gamma(z_n)$ has a grafted euclidean cylinder of length $s= Im(z_n)$. From the definition of the horocycle topology, we have $Im(z_n)\to \infty$ as $n\to \infty$. The sequence of embeddings $i_n: X_n \to X_\infty$ given by the inclusion maps satisfy the conditions of Definition \ref{defn:climit}. In particular, $Im(z_n)\to \infty$ implies $U_n \searrow P$ as required. 
\end{proof}

Consider the Teichm\"{u}ller disk $\mathcal{D}_\gamma :\mathbb{H}\to \mathcal{T}_g$ with basepoint $Y$, where $Y$ is a Jenkins-Strebel surface (Definition \ref{defn:jss}). We also have:

\begin{defn}  Let $Y_\infty$ be the the singular-flat surface obtained by cutting along a horizontal circle on $Y$ and gluing half-infinite euclidean cylinders along the resulting boundary components. Like $X_\infty$, this can also be thought of as a noded surface (with the node $P$ at $\infty$ along the cylinders).
\end{defn}

An argument identical to Lemma \ref{lem:elim} (which we omit) then yields:

\begin{lem}\label{lem:dlim}
Let $z_n \in \mathbb{H}$ ($n\geq 0$) be any sequence such that $z_n\to \infty$  in the horocycle topology. Then $Y_\infty \cup \{P\}$ is the conformal limit of the sequence of Riemann surfaces $\mathcal{D}_\gamma(z_n)$.
\end{lem}

See also \S4.1.6 of \cite{Handbook}. 

 \begin{figure}
  \centering
  \includegraphics{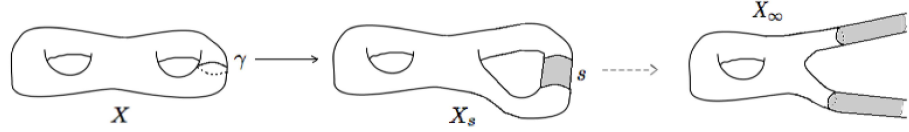}\\
  \caption{The conformal limit along complex-twists as the imaginary part of the parameter $s\to \infty$.}
\end{figure}

\section{Proof of Theorem \ref{thm:thm1}}  

Fix a hyperbolic surface $X$ and a geodesic loop $\gamma$.  Let $l(\gamma)$ denote its length on $X$.\\

Let $\mathcal{E}_\gamma: \mathbb{H}_{\geq 0} \to \mathcal{T}_g$ be the corresponding complex-earthquake deformation. From \S3.3 we have a noded Riemann surface $X_\infty$ such that $\infty \mapsto X_\infty$ defines a  continuous extension $\widehat{\mathcal{E}}_\gamma:\bar{\mathbb{H}}\to \widehat{\mathcal{T}}_g$ in the horocycle topology on $\bar{\mathbb{H}}$.\\

Recall the following theorem of Strebel:

\begin{thm}[Strebel, \cite{Streb}] \label{thm:streb} Let $Z$ be a Riemann surface of genus $g\geq 2$, and $P$ be a collection of $n\geq 1$ points, and   $a_1,a_2,\ldots,a_n$ a tuple of positive reals. Then there exists a meromorphic quadratic differential $q$ on $Z$ with poles at $P$ of order $2$ and  residues $a_1,\ldots a_n$, such that all horizontal leaves (except the critical trajectories) are closed and foliate punctured disks around ${P}$.
\end{thm}

Consider such a  Strebel differential on $X_\infty$ with poles of order two at the two punctures and residue $\frac{l(\gamma)}{2\pi}$ at each. Let $Y_\infty$ refer to the singular flat surface obtained by considering the metric induced by the Strebel differential. This surface comprises two half-infinite euclidean cylinders $\mathcal{C}_1$ and $\mathcal{C}_2$  of circumference $l(\gamma)$ with their boundaries identified by a (piecewise-isometric) interval-exchange map $\mathcal{I}$. Note that we have a conformal homeomorphism:
\begin{equation}\label{eq:mapg}
g:X_\infty \to Y_\infty
\end{equation}
that topologically is the identity map.\\

\textit{Notation.} Let $\mathcal{A}_1$ and $\mathcal{A}_2$ be the grafted half-infinite cylinders on $X_\infty$, and let ${A}_{1,L}$ and ${A}_{2,L}$ denote euclidean sub-cylinders on them of length $L$, adjacent to the boundary $\gamma$. Similarly, let $C_{1,L}$ and $C_{2,L}$ denote a sub-cylinders of length $L$ adjacent to $\gamma$ on $\mathcal{C}_{1}$  and $\mathcal{C}_2$ respectively, in $Y_\infty$.

 \begin{figure}
  \centering
  \includegraphics{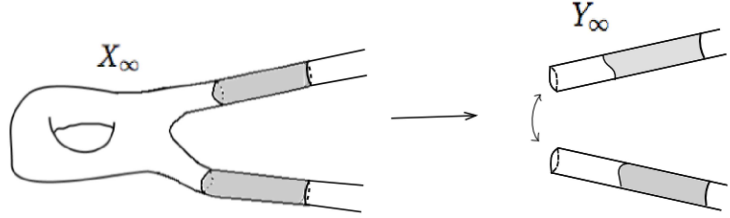}\\
  \caption{The proof of Theorem \ref{thm:thm1} involves adjusting the map (\ref{eq:mapg}) between the conformal limits.}
\end{figure}

\subsection*{The Teichm\"{u}ller disk $\mathcal{D}_\gamma$.}

Consider the surface $Y_0$  obtained by cutting along a meridian of each half-infinite cylinder on $Y_\infty$, and gluing by an isometry such that the cylinder is of length $1$. (That is, in the notation above, by gluing $C_{1, 1/2}$ and $C_{2, 1/2}$.)  Also, this gluing is done without introducing twists:  namely the endpoints of a choice of perpendicular axis on the truncated cylinders are required to match up.\\

The surface $Y_0$ shall form the basepoint of the Teichm\"{u}ller disk $D_\gamma$. Note that $Y_0$ is a Jenkins-Strebel surface (Definition \ref{defn:jss}), obtained by taking a euclidean rectangle ${R}$ of length $l(\gamma)$ and height $1$, identifying the vertical sides by an isometry, and horizontal sides by the interval-exchange map $\mathcal{I}$. By Definition \ref{defn:stsh}, the Teichm\"{u}ller disk:
\begin{equation*}
D_\gamma:\mathbb{H} \to \mathcal{T}_g
\end{equation*}
then takes $i\mapsto Y_0$ and
\begin{equation*}
t + is \mapsto \bigl(\begin{smallmatrix} 1 &0\\ 0& s \end{smallmatrix} \bigr) \cdot  \bigl(\begin{smallmatrix} 1&t\\ 0&1 \end{smallmatrix} \bigr) \cdot R
\end{equation*}
with the corresponding identification of sides of the resulting parallelogram. We shall show this is asymptotic to $\mathcal{E}_\gamma$.

\begin{proof}[Proof of Theorem \ref{thm:thm1}]
Fix a sufficiently small $\epsilon>0$. From the previous discussion in this section it suffices to show that there is an $H>0$ such that if $Im(z) >H$ then there exists a $(1+\epsilon)$-quasiconformal map  $\bar{f:} \mathcal{E}_\gamma(z) \to \mathcal{D}_\gamma(z)$. We build this map as follows:\\

Consider the conformal map $g:X_\infty \to Y_\infty$ as in (\ref{eq:mapg}), and its restrictions $g_1$ and $g_2$ to the half-infinite cylinders $\mathcal{A}_1$ and $\mathcal{A}_2$ respectively. These map to the corresponding half-infinite cylinders $\mathcal{C}_1$ and $\mathcal{C}_2$ on $Y_\infty$ which, by our choice of residue above, have the same circumference. \\

By Lemma \ref{lem:interp}, for each $i=1,2$ we can adjust $g_i$ to a $(1+ \epsilon)$-quasiconformal embedding $f_i:\mathcal{A}_i\to \mathcal{C}_i$  that restricts to $g_i$ on $A_{i,L}$ and to a translation by a distance $e_i$ on $\mathcal{A}_i\setminus A_{i, 2L}$.  (Note that a dilation by factor $c$ is a translation by $\ln c$ in the metric on $\mathbb{D}^\ast$.) Here we choose $L$ sufficiently large such that each $g_i$ is ``twist-free" in $\mathcal{A}_i \setminus A_{i,L}$ (see Lemma \ref{lem:bt}).\\

For any $s>2L$,  cut and glue  the circles $\partial A_{1,s}$ and  $\partial A_{2,s}$  with $t$-twists to get a surface  $\mathcal{E}_\gamma(z)$ where $z := t + i2s$. The image curves $f(\partial A_{1,s})$ and $f(\partial A_{2,s})$ are the circles $\partial C_{1, s+e_1}$ and  $\partial C_{2, s+e_2}$, respectively,  on $Y_\infty$. Cutting and gluing these circles on $Y_\infty$  by  a $t$-twist results in the surface  $\mathcal{D}_\gamma(z^\prime)$ where $z^\prime :=t + i(2s +e_1 + e_2)$. (This is where our choice of the height of $Y_0$ comes in: see Lemma \ref{lem:ttwist}.)\\

\begin{figure}
  \centering
  \includegraphics{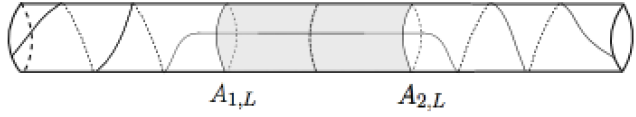}\\
  \caption{The map obtained after gluing the cylinders  is adjusted by post-composing with Dehn-twists supported on the ``twist-free" part (shown shaded).}
\end{figure}

The map $f_i$  on $A_{i,s}$ for $i=1,2$ together with $g$ on the ``hyperbolic part" $X_\infty \setminus (\mathcal{A}_1 \cup \mathcal{A}_2)$  descends to a $(1+\epsilon)$-quasiconformal map $\bar{f}_0$ between the resulting surfaces  $\mathcal{E}_\gamma(z)$  and  $\mathcal{D}_\gamma(z^\prime)$. This is not quite the map we desire, since it may not be homotopic to the identity map. However we can postcompose with a suitable power $N$ of a Dehn-twist $D$ supported on a central cylinder of length $L$,  such that the resulting map is homotopic to the identity (see Figure 8).  Since $D$ is supported on the ``twist-free" part of $g$, the integer $N$ is independent of $\epsilon$ (it depends only on the conformal map (\ref{eq:mapg}). For $L$ large enough, such a  Dehn-twist can be chosen to be $(1+\epsilon)$-quasiconformal by Lemma \ref{lem:untwist}, and the composition $\bar{f}_1 = D^N \circ \bar{f}_0$, is $(1+N\epsilon + \epsilon)$-quasiconformal.\\

A final adjustment is needed to get to $\mathcal{D}_\gamma(z)$: Since $e:=e_1+e_2$ is independent of $\epsilon$, for sufficiently large $s$, we have $\left\vert \frac{s+e}{s}-1\right\vert <\epsilon$, and there is a $(1+ \epsilon)$-quasiconformal vertical affine stretch map  $\sigma: \mathcal{D}_\gamma(z^\prime) \to \mathcal{D}_\gamma(z)$. Postcomposing with this we get a $(1+(N+2)\epsilon)$-quasiconformal map $\bar{f} = \sigma \circ \bar{f}_1: \mathcal{E}_\gamma(z) \to \mathcal{D}_\gamma(z)$ as required (the factor can be absorbed by choosing $\epsilon/(N+2)$ at the beginning of the argument). This completes the construction of $\bar{f}$.\\
 
For the second statement of the theorem, consider the conformal limit $Y_\infty$ of the Teichm\"{u}ller disk $\mathcal{D}_\gamma$ (see Lemma \ref{lem:dlim}). By Theorem 5.4 of \cite{Mondello2}, there is a conformally equivalent ``infinitely-grafted" surface $X_\infty$.  We then have a conformal map $g:X_\infty\to Y_\infty$ as in (\ref{eq:mapg}) and can construct the quasiconformal maps exactly as above.
\end{proof}

\section{Plumbing disks and Theorem \ref{thm:thm2}}

\subsection{Definition of plumbing}
Let $(Z,P)$ be a noded surface. Plumbing is a process of ``opening up" the node, to obtain a closed Riemann surface, as described next (also see  \cite{Kra}).\\

Consider the neighborhood $U_P$ around $P$. As noted earlier, $U_P\setminus P$ is a union of two punctured disks with coordinates $z$ and $w$.\\
Fix $t\in \mathbb{D}^\ast$. From these disks, remove the subdisks $\{ 0<\lvert z\rvert <\sqrt{\lvert t\rvert}\}$ and  $\{0<\lvert w\rvert <\sqrt{\lvert t\rvert}\}$ to get two annuli, and  then glue the circles $\lvert z\rvert, \lvert w\rvert = \sqrt{\lvert t\rvert}$ by the conformal map 
\begin{equation}\label{eq:glue}
w = \frac{t}{z}
\end{equation}
to get a closed surface $\Sigma(t)$.\\
Here, we assume $Z\in \mathcal{C}(X)$ for some $X\in \mathcal{T}_g$ which induces a marking on the plumbed surface (see \S2.1), and hence $\Sigma(t) \in \mathcal{T}_g$.\\

The \textit{plumbing disk} is the map\\
\begin{equation*}
\mathcal{P}:\mathbb{H} \to \mathcal{T}_g
\end{equation*}
that takes
\begin{equation}\label{eq:pdef}
\tau \mapsto \Sigma(e^{i\tau}).
\end{equation}

\begin{figure}[h]
  \centering
  \includegraphics{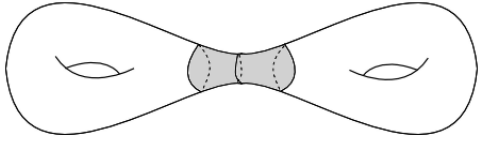}
  \caption{The closed surface obtained by plumbing.}
\end{figure}

\medskip
The surfaces $\mathcal{P}(\tau)$ and $\mathcal{P}(\tau +2\pi)$ differ by a Dehn-twist about the core curve of the plumbed cylinder (\textit{i.e} the gluing in (\ref{eq:glue}) is with an extra twist). In fact $\infty \mapsto (Z,P)$ defines a continuous extension $\mathcal{P}: \mathbb{H}\cup \{\infty\} \to \widehat{\mathcal{T}_g}$ (see the recent work \cite{HubbKoch}). 

\subsection{Plumbing a Strebel surface} Equip $Z\setminus P$ with a Strebel differential $q_\infty$ with residue $1$ at the paired punctures (see Theorem \ref{thm:streb}). Then there is a canonical choice of coordinate neighborhoods  $U_0$ around the punctures: namely, the corresponding half-infinite cylinders $\mathcal{C}_1,\mathcal{C}_2$ (in which  $q_\infty = dz^2$ and $dw^2$ respectively). We call this the \textit{Strebel coordinate} neighborhood of the node $P$. Let $U_0 = \mathcal{C}_1\sqcup \mathcal{C}_2$ (note that $Z\setminus U_0$ has empty interior!). The corresponding plumbing disk  shall be denoted by $\mathcal{P}_0$.\\

Now let $Y_0$ be the closed Riemann surface obtained by, as in \S4,  truncating the half-infinite cylinders  $\mathcal{C}_1$ and $\mathcal{C}_2$ and gluing by a twist-free isometry such that the resulting Jenkins-Strebel surface has height $1$. Consider the Teichm\"{u}ller disk  $\mathcal{D}_0$ with basepoint $Y_0$.

\begin{lem}\label{lem:pd1}
$\mathcal{P}_0$ and  $\mathcal{D}_0$  coincide, that is, $\mathcal{P}_0(z) = \mathcal{D}_0(z)$ for all $ z\in \mathbb{H}$.
\end{lem}
\begin{proof}
Recall the conformal identification of each half-infinite cylinder $\mathcal{C}_1$ and $\mathcal{C}_2$ with the punctured disk $\mathbb{D}^\ast$ equipped with the conformal metric $\frac{\lvert dz\rvert}{\lvert z\rvert}$. For each $i=1,2$ the subdisk $B_r$ in these coordinates corresponds to the half-infinite subcylinder  $\mathcal{C}_i \setminus C_{i,L}$ where $L = \ln \frac{1}{r}$.\\
The plumbed  surface $\mathcal{P}_0(t + is) = \Sigma(e^{- s}e^{i t})$ is obtained by excising the subdisk of radius $e^{-s/2}$ and gluing the resulting boundary by a rotation by angle $t$. In the flat (cylindrical) metric, this corresponds to removing $\mathcal{C}_{i,s/2}$ from $\mathcal{C}_i$ and gluing the boundary components of the resulting surface by a $t$-twist (see Definition \ref{defn:ttwist}). Using Lemma \ref{lem:ttwist} and Definition \ref{defn:stsh} the resulting surface is $\mathcal{D}_0(t + is)$, that is, $Y_0$ postcomposed with a $t$-shear and a $s$-stretch (in that order). \end{proof}

\subsection{Comparing plumbing disks}

Let $(Z,P)$ be a noded Riemann surface, and consider a neighborhood $U_P$ of the node different from the Strebel coordinate neighborhood $U_0$ as in the previous section. Let $\mathcal{P}:\mathbb{H}\to \mathcal{T}_g$ be the corresponding plumbing disk.

\begin{lem}\label{lem:pd2} Given $\mathcal{P}$ and $\mathcal{P}_0$ as above there exists an $e\in \mathbb{R}$ such that  for any $\epsilon>0$ there is an $H>0$ such that 
\begin{equation*}
d_\mathcal{T}(\mathcal{P}(z), \mathcal{P}_0(z + ie))<\epsilon
\end{equation*}
for all $Im(z) >H$.
\end{lem}

\begin{figure}
  \centering
  \includegraphics{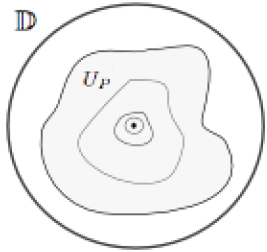}\\
  \caption{By standard distortion theorems, small circles in the coordinate $U_P$ are almost-circular in the Strebel coordinate. This principle underlies the proof of Lemma \ref{lem:interp}, that is used in Lemma \ref{lem:pd2}.}
\end{figure}

\begin{proof}
As before, identify each of the two components of $U_P\setminus P$ with a punctured disk $\mathbb{D}^\ast$. For some $0<r_0<1$ sufficiently small, the inclusion map $i:B_{r_0}\to Z$ maps into the Strebel coordinate $U_0$ as in the previous section. Via the conformal identification of each component of $U_0$ with $\mathbb{D}^\ast$, one obtains for each $i=1$ and $2$, a conformal map $\phi_i:B_{r_0} \to \mathbb{D}$ such that $\phi_i(0)=0$ and $\phi_i^\prime(0)=c_i$.  (As usual we extend the conformal map across the puncture.) We shall assume that $r_0$ is sufficiently small such that $B_{r_0}$ is contained in the ``twist-free" neighborhood of $0$ for $\phi_i$ (see Lemma \ref{lem:bt}.) \\

Note that the derivatives $c_i$ depend only on the choice of $U_P$ and are independent of $r_0$. Moreover, by postcomposing the conformal identification with the punctured disk with the right ``rotation", one can assume $c_i\in \mathbb{R}$. \\

For $i=1,2$ apply Lemma \ref{lem:interp} to obtain a $(1+\epsilon)$-quasiconformal map $f_i:B_{r_0} \to \mathbb{D}$ that interpolates between $\phi_i$ on $\partial B_{r_0}$ and a dilatation $z\mapsto c_iz$ on $B_{r_1}$, for some sufficiently small $r_1<r_0$.\\

Let $r<r_1$. The map $f_i$ then takes $B_{r}$ to the subdisk $B_{c_ir}$. Consider the plumbed surface obtained by excising $B_r \subset\mathbb{D}$ (in the coordinate $U_P$) and gluing by a $t$-rotation. By  definition (see \S5.1) this surface is $\mathcal{P}(t + is)$ for $ s = 2\ln \frac{1}{r}$. (The factor $2$ appears as $r$ is the square-root of the plumbing parameter as in (\ref{eq:glue}).)\\
Similarly, the plumbed surface obtained by excising $B_{c_1r}$ and $B_{c_2r}$ (in the two disks of the coordinate $U_0$) and gluing by a $t$-rotation is identical to $\mathcal{P}_0(t +  i(s + e))$ where $e= \ln c_1 + \ln c_2$ and $s$ is as above.\\

The maps $f_i$ restricted to $B_{r_0} \setminus B_{r}$ (on each component of $U_P\setminus P$) together with the identity map on the rest of the surface extends across the boundary circles $\partial B_r$ (since the amount of twist is the same). This  defines a $(1+\epsilon)$-quasiconformal map $\hat{f}: \mathcal{P}(t + is) \to \mathcal{P}_0(t +  i(s + e))$ as required. Note that since $r<r_0$ we have  $s > H$ where $H:=2\ln \frac{1}{r_0}$.  \end{proof}

\subsection{Concluding the proof}
\begin{proof}[Proof of Theorem \ref{thm:thm2}]

Given a noded Riemann surface $(Z,P)$ and a plumbing disk $\mathcal{P}:\mathbb{H} \to \mathcal{T}_g$ consider the plumbing disk $\mathcal{P}_0$ determined by the Strebel coordinates and the corresponding Teichm\"{u}ller disk $\mathcal{D}_0$ as in \S5.2.  This $\mathcal{D} := \mathcal{D}_0$ shall be the Teichm\"{u}ller disk in the statement of the theorem. \\

Fix a sufficiently small $\epsilon>0$. By Lemma \ref{lem:pd1} ,  $\mathcal{P}_0$ =  $\mathcal{D}_0$ on $\mathbb{H}$. This gives the second statement of the theorem. By Lemma \ref{lem:pd2}, we then have:
\begin{equation}\label{eq:peq1}
d_\mathcal{T}(\mathcal{P}(z), \mathcal{P}_0(z + ie))= d_\mathcal{T}(\mathcal{P}(z), \mathcal{D}_0(z + ie))<\epsilon
\end{equation}
when $Im(z)$ is sufficiently large.\\
 Since $e$ is independent of $z$ and $\epsilon$ (it depends only on the choice of $\mathcal{P}$), we have that $\beta := \left\vert \frac{Im(z+ie)}{Im(z)} \right\vert $ satisfies:
 \begin{equation*}
 \left\vert \beta -  1\right\vert  < \epsilon 
 \end{equation*}
 when $Im(z)$ is sufficiently large. A vertical affine stretch by a factor $\beta$ of the Jenkins-Strebel surface then defines a $(1+ \epsilon)$-quasiconformal map from $ \mathcal{D}_0(z + ie)$ to $ \mathcal{D}_0(z)$. Hence we have:
 \begin{equation}\label{eq:peq2}
 d_\mathcal{T}(\mathcal{D}_0(z + ie), \mathcal{D}_0(z))\leq \frac{1}{2} \ln (1+\epsilon) < \epsilon
 \end{equation}
since $\epsilon$ is sufficiently small, and
\begin{equation*}
 d_\mathcal{T}(\mathcal{P}(z), \mathcal{D}_0(z))<2\epsilon
 \end{equation*}
 follows from (\ref{eq:peq1}), (\ref{eq:peq2}) and the triangle inequality (we can choose a smaller $\epsilon$ to absorb the constant factor). \end{proof}

\bibliographystyle{amsplain}
\bibliography{qmref}

\end{document}